\documentclass{amsart}

\usepackage{amssymb}
\usepackage{amsmath,amscd}
\usepackage{graphicx}
\usepackage{color,hyperref}
\usepackage{tikz}
\usetikzlibrary{matrix,arrows}

\newcommand{\mf}{\mathfrak}
\newcommand{\mc}{\mathcal}

\newcommand{\R}{\mathbf R}
\newcommand{\C}{\mathbf C}
\newcommand{\Q}{\mathbf Q}
\newcommand{\Z}{\mathbf Z}
\newcommand{\F}{\mathbf F}

\newcommand{\Fix}{\textnormal{Fix}}

\numberwithin{equation}{section}

\theoremstyle{plain}
\newtheorem{theorem}{Theorem}[section]
\newtheorem{lemma}[theorem]{Lemma}

\theoremstyle{definition}

\newtheorem{remark}[theorem]{Remark}

\begin{document}

\title[Bivariate polynomial mappings]{Bivariate polynomial mappings associated 
with simple complex Lie algebras}

\author{\"{O}mer K\"{u}\c{c}\"{u}ksakall{\i}}
\address{Middle East Technical University, Mathematics Department, 06800 Ankara,
Turkey.}
\email{komer@metu.edu.tr}

\date{\today}

\begin{abstract}
There are three families of bivariate polynomial maps associated with the 
rank-$2$ simple complex Lie algebras $A_2, B_2 \cong C_2$ and $G_2$. It is 
known that the bivariate polynomial map associated with $A_2$ induces a 
permutation of $\F_q^2$ if and only if $\gcd(k,q^s-1)=1$ for $s=1, 2, 3$. 
In this paper, we give similar criteria for the other two families. As an 
application, a counterexample is given to a conjecture posed by Lidl and Wells 
about the generalized Schur's problem.
\end{abstract}



\maketitle

%\iffalse
%\fi

\section*{Introduction}
A polynomial map $f:\C^n\rightarrow\C^n$ of degree greater than one is called 
\textit{integrable} if there exists a polynomial map $g:\C^n\rightarrow\C^n$ of 
degree greater than one such that $f$ and $g$ commute, i.e. $f\circ g = g\circ 
f$, and the set of iterations of $f$ and $g$ are disjoint. Integrable maps play 
an important role in the theory of dynamical systems because they show an 
unusual degree of symmetry \cite{veselov-survey}. In the case $n=1$, a full 
description of integrable polynomials was given by Julia \cite{julia}, Fatou 
\cite{fatou} and Ritt \cite{ritt}. An integrable polynomial map 
$f:\C\rightarrow\C$ can be transformed by a linear change of variables to the 
form $f=z^n$ or $f=\pm T_n(z)$, where $T_n(z)=\cos(n \arccos z )$ is the 
Chebyshev polynomial.

There is a question in the theory of finite fields which has a similar answer. 
A polynomial $f(x)\in\Z[x]$ is called \textit{exceptional} if $f(x)$ induces a 
permutation of an infinite number of finite fields $\F_p$ where $p$ is prime. It 
is well known that a polynomial is exceptional if and only if it is a 
composition of linear polynomials, power maps and the Chebyshev polynomials. 
One side of this statement is relatively easier to prove since the $k$-th power 
map and the $k$-th Chebyshev polynomial induce a permutation of $\F_q$ if and 
only if $\gcd(k,q-1)=1$ and $\gcd(k,q^2-1)=1$, respectively \cite{lidlnied}. 
The other side of this classification is known as Schur's problem and  
proved by Fried \cite{fried}. Other proofs have been given by Turnwald 
\cite{turnwald} and M\"{u}ller \cite{muller}.

Let $\mathbf{P}^1(\C)$ be the projective space of dimension one. Apart from the 
power maps and Chebyshev polynomials, there is one more family of rational maps 
on $\mathbf{P}^1(\C)$ which satisfies the commuting relation $f \circ 
g = g \circ f$ \cite{ritt}. It is the family of Latt\`{e}s maps induced by 
isogenies of an elliptic curve $E$.  In our previous 
work \cite{kucuksakalli}, using the underlying elliptic curve group 
structure, we gave a criterion when a Latt\`{e}s map induces a permutation of 
$\mathbf{P}^1(\F_q)$. In the theory of dynamical systems, especially in its 
arithmetical aspects, the underlying algebraic structure plays an important 
role. For example, see Silverman \cite{sil-dyn}.

In this paper, we pay attention to the bivariate polynomial mappings associated 
with the rank-2 simple complex Lie algebras. We fix some notation first. Let 
$\mf{g}$ be a complex Lie algebra of rank $n$ and $\mf{h}$ its Cartan 
subalgebra, $\mf{h}^*$ its dual space, $\mc{L}$ a lattice of weights in 
$\mf{h}^*$ generated by the fundamental weights $\omega_1, \ldots, \omega_n$, 
and $L$ the dual lattice in $\mf{h}$. Veselov defines the mapping 
$\varPhi_\mf{g}:\mf{h}/L \rightarrow \C^n$, $\varPhi_\mf{g}(\varphi_1,\ldots, 
\varphi_n)$,
\[\varphi_k=\sum_{w\in W} e^{2\pi i w (\omega_k)}\] 
where $W$ is the Weyl group, acting on the space $\mf{h}^*$. Veselov shows that 
there exist a family of polynomial mappings associated with each simple complex 
Lie algebra with nice dynamical properties. Hofmann and Withers give the same 
result independently somewhat later.

\begin{theorem}[\cite{veselov},\cite{hoffwith}]\label{veselov}
With each simple complex Lie algebra of rank $n$, there is an associated an 
infinite series of integrable polynomial mappings $P_\mf{g}^k$, determined from 
the conditions
\[ \varPhi_\mf{g}(k\mathbf{x})=P_\mf{g}^k(\varPhi_\mf{g}(\mathbf{x})). \]
All coefficients of the polynomials defining $P_\mf{g}^k$ are integers.
\end{theorem}

The commutativity of $P_\mf{g}^k$ follows from their definition:
\[ P_\mf{g}^k \circ P_\mf{g}^l = P_\mf{g}^{kl} = P_\mf{g}^l \circ P_\mf{g}^k.\]
The fact that they are polynomials follow from Chevalley's theorem which 
implies that the functions $\varphi_k$ freely generate an algebra of 
exponential invariants of a Weyl group $W$ \cite{veselov-survey}.

For $n=1$, there is a unique simple algebra $A_1$ of rank one. We have 
$\varphi_1=e^{2\pi i x}+e^{-2\pi i x} = 2\cos(2\pi x)$, and the polynomials 
$P_{A_1}^k$ are conjugate to the Chebysev polynomials. Indeed $P_{A_1}^k = 
D_k(x)$ is the family of Dickson polynomials. The Dickson polynomial $D_k$ 
satisfies the relation $D_k(y+1/y)=(y^k+1/y^k)$ for an indeterminate $y$ and
induces a permutation of the finite field $\F_q$ if and only if 
$\gcd(k,q^2-1)=1$ \cite{lidlnied}.

There are three distinct rank-2 simple complex Lie algebras, namely $A_2, B_2 
\cong C_2$ and $G_2$. The case $A_2$ was considered by Lidl and Wells as a part 
of a general result for $A_n$ \cite{lidlwells}. The polynomial $P_{A_2}^k$ 
induces a permutation of $\F_q^2$ if and only if $\gcd(k,q^s-1)=1$ for 
$s=1,2,3$. See Remark~\ref{rmrk-an}. We provide similar criteria for the 
families associated with the other two Lie algebras $B_2 \cong C_2$ and $G_2$, 
see Theorem~\ref{mainb2} and Theorem~\ref{maing2}, respectively.

The organization of this paper is as follows: In Section~\ref{sec:A}, we review 
the construction given by Lidl and Wells and its relation with the simple Lie 
Algebra $A_n$. In Sections~\ref{sec:B}~and~\ref{sec:G}, we investigate 
the bivariate polynomial maps associated with $B_2\cong C_2$ and $G_2$, 
respectively. We analyze the fixed points of these maps over complex numbers and 
obtain a one-to-one correspondence with $\F_q^2$ that is given by reduction 
modulo a certain prime ideal. We also explain why these examples of bivariate 
polynomial mappings disprove a conjecture posed by Lidl and Wells.

\section{The family associated with $A_n$}\label{sec:A}
Lidl and Wells give a generalization of Chebyshev maps to higher dimensions 
\cite{lidlwells}. Even though the underlying structure can be realized to be 
$A_n$, their construction is elementary. Their main tool is the fundamental 
theorem on symmetric polynomials. Lidl and Wells consider the polynomial 
equation 
\begin{equation}\label{eqn:A}
 t^{n+1}+x_1t^n+\ldots+x_nt+b=0
\end{equation}
with coefficients $x_1,\ldots,x_n, b \in\C$ and roots $t_1,\ldots,t_{n+1} \in 
\C$. The roots $t_i$ are not necessarily distinct. Note that $ b= (-1)^{n+1} 
t_1\cdots t_{n+1}$. Let $k$ be a positive integer. Consider the polynomial 
equation 
\[ t^{n+1}+\tilde x_1t^n+\ldots+\tilde x_nt+\tilde b=0 \]
with roots $t_1^k,\ldots,t_{n+1}^k$. Note that $\tilde b=(-1)^{(n+1)(k+1)}b$. 
It follows from the fundamental theorem on symmetric polynomials that there are 
integral polynomials $g_1^{(k)},\ldots,g_n^{(k)}$ such that $ g_i^{(k)} 
(x_1,\ldots,x_n,b) = \tilde{x}_i$ for all $i=1,\ldots,n$. Thus one can 
consider the polynomial vector
\[ g(n,k,b)=(g_1^{(k)}(x_1,\ldots,x_n,b),\ldots,g_n^{(k)}(x_1,\ldots,x_n,b)). \]
This polynomial vector can be regarded as a map from $\C^n$ to $\C^n$. Let 
$\F_q$ be a finite field of characteristic $p$. If $b$ is an integer, then each 
component of $g(n,k,b)$ is in $\Z[x_1,\ldots,x_n]$ and $g(n,k,b)$ induces a 
mapping $\bar{g}(n,k,b)$ from $\F_q^n$ to $\F_q^n$. The main result of Lidl and 
Wells is the following:

\begin{theorem}[\cite{lidlwells}]\label{lidlwellsmain}
 If $b\not\equiv 0\pmod{p}$, then the mapping $\bar{g}(n,k,b)$ is a permutation 
if and only if $\gcd(k,q^s-1)=1$ for all $s=1,2,\ldots,n+1$. Moreover 
$\bar{g}(n,k,0)$ is a permutation if and only if $\gcd(k,q^s-1)=1$ for all 
$s=1,2,\ldots,n$.
\end{theorem}

Note that the case $n=1$ of this theorem is the well known criteria for 
Chebyshev polynomials and power maps to be permutations. It follows from this 
theorem and Dirichlet's theorem on primes in arithmetic progression, the 
polynomial $g(1,k,b)$ induces a permutation of $\F_q$ for an infinite 
number of finite fields when $\gcd(k,6)=1$. This fact has a remarkable converse. 
If $f(x)\in\Z[x]$ permutes $\F_p$ for an infinite number of primes $p$, then it 
is a composition of linear polynomials and the polynomials $g(1,k,b)$. 
This is known as Schur's problem and proved by Fried \cite{fried}. Other proofs 
have been given by Turnwald \cite{turnwald} and M\"{u}ller \cite{muller}.

Inspired by this state of art, Lidl and Wells make the following conjecture in  
their manuscript \cite{lidlwells}: If $h_1(x_1, \ldots, x_n), \ldots, h_n(x_1, 
\ldots, x_n)$ are integral polynomials such that 
\[ H:(x_1,\ldots,x_n)\mapsto (h_1(x_1,\ldots,x_n),\ldots,h_n(x_1,\ldots,x_n))\]
is a permutation of $\F_p^n$ for an infinite number of primes, then $H$ is a 
composition of linear polynomial vectors and polynomial vectors $g(k,n,b)$ 
where $k$ and $b$ are various integers. We will show in the latter sections 
that the families of bivariate polynomials associated with $B_2\cong C_2$ and 
$G_2$ constitute counterexamples to this conjecture. 

The construction of Lidl and Wells is related with the Lie algebra $A_n$. Given 
$\mathbf{x}=(x_1,\ldots,x_n)$, set $x_{0}=-(x_1+\ldots+x_n)$. Let $\sigma_i$ 
denote the $i$th elementary symmetric polynomial. The maps $\varphi_i$ of 
Theorem~\ref{veselov} turn out to be
\[ \varphi_i(\mathbf{x}) = \sigma_i(e^{-2\pi i x_0},\ldots,e^{-2\pi i 
x_{n}}). \]
See \cite{hoffwith} for the details. We have $\varPhi_{A_n} = (\varphi_1, 
\ldots, \varphi_n)$ and the associated infinite series of integrable polynomial 
mappings $P_{A_n}^k$ are determined from the conditions
\[ \varPhi_{A_n}(k\mathbf{x})=P_{A_n}^k(\varPhi_{A_n}(\mathbf{x})). \]
In equation (\ref{eqn:A}), we observe that $x_i = (-1)^i \sigma_i(t_1, \ldots, 
t_{n+1})$ for each $i\in\{1,\ldots,n\}$. In order to deal with plus or minus 
signs, we define 
\[L(x_1,x_2,\ldots,x_n) = ((-1)^1x_1, (-1)^2x_2, \ldots, (-1)^nx_n).\]
Note that $\sigma_{n+1}(e^{-2\pi i x_0},\ldots,e^{-2\pi i x_{n}})=1$ and 
$b=(-1)^{n+1} \sigma_{n+1} (t_1, \ldots, t_{n+1})$. The relation between the 
bivariate maps $P_{A_n}^k$ and $g(n,k,b)$ is given by
\[P_{A_n}^k = L\circ g(n,k,(-1)^{n+1}) \circ L^{-1}.\]
\begin{remark}\label{rmrk-an}
 Theorem~\ref{lidlwellsmain} remains true for $P_{A_n}^k$ since it is conjugate 
to the mapping $g(n,k,(-1)^{n+1})$ under a linear transformation. More precisely 
$P_{A_n}^k$ induces a permutation of $\F_q^2$ if and only if $\gcd(k,q^s-1)=1$ 
for $s=1,2,\ldots,n+1$.
\end{remark}

\section{The family associated with $B_2 \cong C_2$}\label{sec:B}
We refer to \cite{lie} for a nice introduction to the theory of Lie algebras. 
Let $\{\alpha_1, \alpha_2\}$ be a choice of simple roots for the Lie algebra 
$B_2$ with Cartan matrix 
\[\left[\begin{array}{cc}2 & -2\\-1 & 2 \end{array}\right].\]
The transpose of this matrix transforms the fundamental weights into the 
fundamental roots. We have
\begin{align*}
 \alpha_1&=2\omega_1-\omega_2,\\
 \alpha_2&=-2\omega_1+2\omega_2.
\end{align*}
The function $\varPhi_{B_2}=(\varphi_1,\varphi_2)$ of Theorem~\ref{veselov} is 
obtained by the action of the Weyl group on the fundamental weights $\omega_1$ 
and $\omega_2$. The functions $\varphi_1$ and $\varphi_2$ turn out to 
be
\begin{align*}
 \varphi_1(\sigma,\tau) &= e^{2\pi i \sigma}+e^{-2\pi i \sigma}+e^{2\pi i 
\tau}+e^{-2\pi i \tau}\\
 \varphi_2(\sigma,\tau) & = e^{2\pi i (\sigma+\tau)}+e^{2\pi i 
(\sigma-\tau)}+e^{2\pi i (-\sigma+\tau)}+e^{2\pi i (-\sigma-\tau)}
\end{align*}
If $(\sigma,\tau)\in\R^2$, then we can simply write
\begin{align*}
 \Phi_{B_2}(\sigma,\tau) = (2\cos(2\pi \sigma)+2\cos(2\pi \tau),2\cos(2\pi 
\sigma)2\cos(2\pi \tau)).
\end{align*}
Hofmann and Withers call this map the generalized cosine function for the 
underlying Lie algebra \cite{hoffwith}. Theorem~\ref{veselov} implies that 
there are bivariate polynomial mappings $P_{B_2}^k$, determined from the 
conditions 
$\varPhi_{B_2}(k\mathbf{x})=P_{B_2}^k(\varPhi_{B_2}(\mathbf{x}))$ where 
$\mathbf{x}=(\sigma,\tau)$. For simplicity, let us put 
\[ \mc{B}_k := P_{B_2}^k. \]
These maps satisfy the composition property $\mc{B}_{kl} = \mc{B}_{k} \circ 
\mc{B}_{l} = \mc{B}_{l} \circ \mc{B}_{k}$ by their definition. The first few 
examples of these polynomials are:
\begin{align*}
 \mc{B}_0(x,y)&=(4,4),\\
 \mc{B}_1(x,y)&=(x,y),\\
 \mc{B}_2(x,y)&=(x^2-2y-4,y^2-2x^2+4y+4),\\
 \mc{B}_3(x,y)&=(x^3-3xy-3x,y^3-3x^2y+6y^2+9y).
\end{align*}
There is a recurrence relation satisfied by these maps from which it is 
straightforward to calculate further $\mc{B}_k$ \cite{withers}. If 
$\mc{B}_k=(f_k,g_k)$, then
\begin{align*}
 f_{k+4}&=x(f_{k+3}+f_{k+1})-(2+y)f_{k+2}-f_k,\\
 g_{k+4}&=y(g_{k+3}+g_{k+1})-(x^2-y-2)g_{k+2}-g_k.
\end{align*}

Consider the map $\phi(t_1,t_2) = (t_1+1/t_1,t_2+1/t_2)$ and $\psi(u_1,u_2) = 
(u_1+u_2,u_1u_2)$. The bivariate map $\mc{B}_k$ fits into the following 
commutative diagram:
\begin{center}
\begin{tikzpicture}
  \matrix (m) [matrix of math nodes,row sep=1.5em,column sep=18em,minimum 
width=2em] {
     {\C^*}^2 & {\C^*}^2 \\
     \C^2 & \C^2 \\
     \C^2 & \C^2 \\};
  \path[-stealth]
    (m-1-1) edge node [left] {$\phi$} (m-2-1)
            edge node [above] 
{$(t_1,t_2)\mapsto(t_1^k,t_2^k)$} (m-1-2)
    (m-2-1) edge node [left] {$\psi$} (m-3-1)
    edge node [above] {$(u_1,u_2)\mapsto(D_k(u_1),D_k(u_2))$} 
(m-2-2)
    (m-1-2) edge node [right] {$\phi$} (m-2-2)
    (m-3-1) edge node [above] {$(x,y)\mapsto \mc{B}(x,y)$} (m-3-2)
    (m-2-2) edge node [right] {$\psi$} (m-3-2);
\end{tikzpicture}\end{center}
Recall that the $k$th Dickson polynomial satisfies $D_k=P_{A_1}^k$. Note that 
$\varPhi_{B_2}(\sigma,\tau)=\psi(\phi(e^{2\pi i \sigma}, e^{2\pi i \tau}))$.
The commutativity of this diagram now follows from Theorem~\ref{veselov}. An 
important consequence of this diagram is the following:
\begin{lemma}\label{lemmab2}
 Let $q$ be a power of a prime $p$. Then $\mc{B}_q(x,y)\equiv(x^q,y^q) 
\pmod{p}$.
\end{lemma}
\begin{proof}
The coefficients of the Dickson polynomials $D_k(x)$ can be computed using the 
following formula:
\[ D_k(x) = \sum_{i=0}^{\lfloor k/2\rfloor} \frac{k(-1)^i}{k-i}\binom{k-i}{i}
 x^{k-2i}.\]
Let $q$ be a power of a prime $p$. It is easily verified using this formula 
that $D_q(x)\equiv x^q \pmod{p}$. It follows that
 \begin{align*}
  \psi({D}_q(u_1),{D}_q(u_2)) &\equiv \psi(u_1^q,u_2^q) \pmod{p}\\
  &\equiv(u_1^q+u_2^q,(u_1^q u_2^q)) \pmod{p}\\
  &\equiv((u_1+u_2)^q,(u_1u_2)^q) \pmod{p}.
 \end{align*}
Therefore $\mc{B}_q(x,y)\equiv(x^q,y^q) \pmod{p}$.
\end{proof}

In the theory of dynamics, the (forward) orbit of a point $\alpha\in S$ under 
$f:S\rightarrow S$ is the set $ \mc{O}_f(\alpha)=\{ f^n(\alpha) \mathrel| 
n\geq 0 \}$ by definition. A point $\alpha$ is said to have a bounded orbit 
under $f$ if the set $\mc{O}_f(\alpha)$ is bounded. See Silverman 
\cite{sil-dyn} for a nice introduction to dynamical systems with an emphasis on 
their arithmetical aspects. 

Let $k\geq 2$ be an integer. Consider the power map $z\mapsto z^k$ on $\C^*$. 
The set of points with bounded orbits under the power map is the unit circle 
$\{ z \mathrel| |z|=1 \}$. Using the commutative diagram above, we see that the 
set of points with bounded orbits under $\mc{B}_k:\C^2 \rightarrow \C^2$ is 
$\Delta_{B_2}=\{ \psi(\phi(z_1,z_2)) \mathrel| |z_1|=1, |z_2|=1 \}$. This set 
can be described with the help of $\varPhi_{B_2}$ as well. More precisely, we 
have
\[ \Delta_{B_2}=\{ \varPhi_{B_2}(\sigma,\tau) \mathrel| \sigma,\tau\in\R\}. \]

A point $\alpha$ that is fixed under $f$ has a bounded orbit since 
$\mc{O}_f(\alpha)$ consists of a single point. As a result the set 
$\Fix(\mc{B}_k)=\{\alpha\in \C^2 \mathrel| \mc{B}_k(\alpha)=\alpha\}$ is 
contained in $\Delta_{B_2}$. In other words a point that is fixed under 
$\mc{B}_k:\C^2 \rightarrow \C^2$ is of the form $\varPhi_{B_2}(\sigma,\tau)$ for 
some $\sigma,\tau\in \R$. The set $\Delta_{B_2}$, which is shown in 
Fig.~\ref{fig:funB2}, is contained in $\R^2$.

\begin{figure}[htbp]
    \centering
 \includegraphics[scale=0.8]{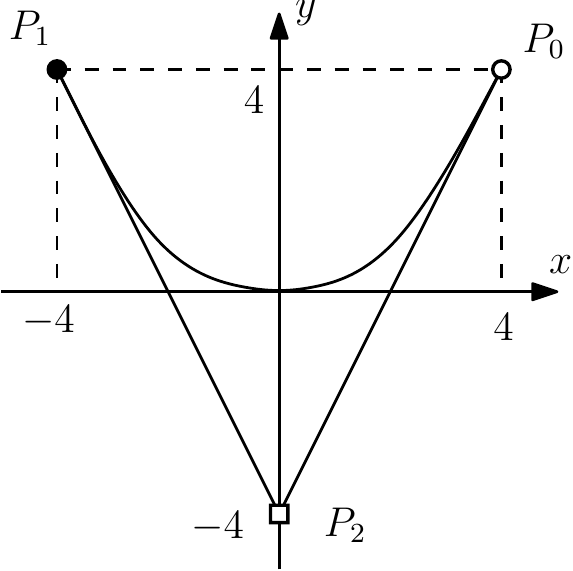}
    \caption{The set $\Delta_{B_2}$.}
    \label{fig:funB2}
\end{figure}

There are three corner points of $\Delta_{B_2}$, namely $P_0=(4,4), P_1=(-4,4)$ 
and $P_2=(0,-4)$. The set  $\Delta_{B_2}$ is bounded by the lines $y+4\pm 2x=0$ 
and the parabola $4y=x^2$. We want to find a fundamental region in 
$\sigma\tau$-plane whose elements are in one-to-one correspondence with the
elements of $\Delta_{B_2}$ under $\varPhi_{B_2}$. If $(\sigma,\tau) 
\equiv (\sigma',\tau')\pmod{\Z^2}$, then it is easy to see that $\varPhi_{B_2} 
(\sigma,\tau) = \varPhi_{B_2} (\sigma',\tau')$. Thus it is enough to consider $0 
\leq \sigma,\tau \leq 1$. Moreover there are extra symmetries coming from the 
action of the Weyl group. Observe that $\varPhi_{B_2}(\sigma,\tau)$ is equal to 
any one of the following eight expressions:
\[\begin{array}{ll|ll}
\text{I} &\varPhi_{B_2}(\sigma,\tau) &\text{V} &\varPhi_{B_2}(\tau,\sigma)\\
\text{II} &\varPhi_{B_2}(-\sigma,\tau) &\text{VI} &\varPhi_{B_2}(-\tau,\sigma)\\
\text{III} &\varPhi_{B_2}(\sigma,-\tau) &\text{VII} 
&\varPhi_{B_2}(\tau,-\sigma)\\
\text{IV} &\varPhi_{B_2}(-\sigma,-\tau) &\text{VIII} 
&\varPhi_{B_2}(-\tau,-\sigma)
\end{array}\]
Under these symmetries, the square $0 \leq \sigma,\tau \leq 1$ can be separated 
into eight subtriangles. This is shown in Fig.~\ref{fig:RB2}. Define
\[ R_{B_2}=\{ (\sigma,\tau)\in\R^2 \mathrel| 0 \leq \sigma \leq 1/2 
\text{ and } \sigma \leq \tau \leq 1/2  \}. \]
Note that the restricted function $\varPhi_{B_2}:R_{B_2} \rightarrow 
\Delta_{B_2}$ is one-to-one and onto. Set $\tilde{P}_0= (0,0), \tilde{P}_1= 
(1/2,1/2)$ and $\tilde{P}_2= (0,1/2)$. Then $\varPhi_{B_2} (\tilde{P}_i)=P_i$ 
for each $i\in\{1,2,3\}$. This correspondence (and more) is symbolized by the 
use of different marks, such as circles, disks and squares.

\begin{figure}[htbp]
    \centering
 \includegraphics[scale=0.8]{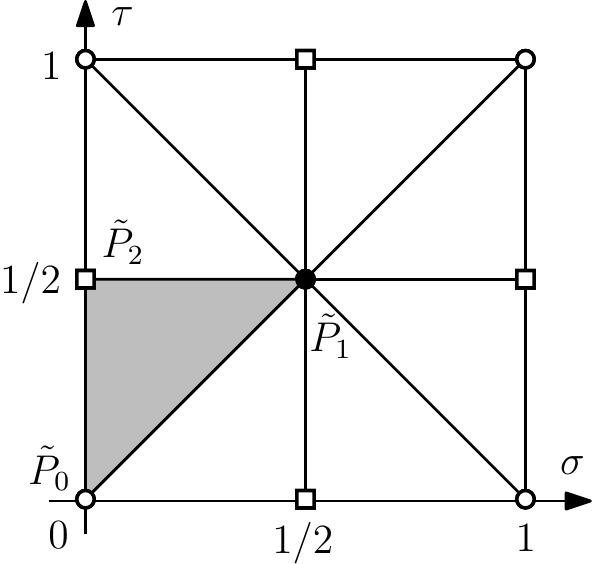}
    \caption{The fundamental region $R_{B_2}$.}
    \label{fig:RB2}
\end{figure}

Now, we are ready analyze the set of fixed points under the bivariate map 
$\mc{B}_k$. 

\begin{theorem}\label{fixb2}
Let $k \geq 2$ be a fixed integer. Then 
\[ \Fix(\mc{B}_k)=\left\{ \varPhi_{B_2}\left( \frac{d}{k\pm1}, \frac{e}{k\pm1} 
\right):d,e\in\Z \right\} \cup \left\{ \varPhi_{B_2}\left( \frac{d}{k^2\pm1}, 
\frac{\pm k d}{k^2\pm1} \right):d\in\Z \right\}\]
where the signs of $k^2\pm 1$ terms agree.
\end{theorem}
\begin{proof}
Let $\alpha = \varPhi_{B_2}(\sigma,\tau)$ be a fixed point under the map 
$\mc{B}_k$. Then we have $\varPhi(k\sigma,k\tau) = \varPhi(\sigma,\tau)$.
In the statement of the theorem, there are eight different choices of sign, 
each one of which corresponds to one of the eight symmetries above. We will 
prove the theorem for one of these. The others are similar. Suppose that 
$(k\sigma,k\tau) \equiv (-\tau,\sigma)$ modulo $\Z^2$. This is the type VI. In 
this case 
\[k^2\sigma \equiv -k\tau \equiv -\sigma \pmod{\Z}.\] 
It follows that $\sigma=d/(k^2+1)$ and therefore $\tau= -k d/(k^2+1)$ for 
some integer $d$. Thus $\alpha$ is of the form 
$\varPhi_{B_2}(d/(k^2+1),-kd/(k^2+1))$.
\end{proof}

The following theorem gives the cardinality of the set of fixed points under 
the bivariate map $\mc{B}_k:\C^2\rightarrow \C^2$.
\begin{theorem}\label{uchib2}
Let $k\geq2$ be a fixed integer. Then $|\Fix(\mc{B}_k)|=k^2$.
\end{theorem}
\begin{proof}
We follow the idea of Uchimura \cite{uchimura}. The fundamental region 
$R_{B_2}$ is a closed bounded domain. Divide $R_{B_2}$ into $k^2$ subtriangles 
$T_1,\ldots,T_{k^2}$ such that each one of them is mapped onto $R_{B_2}$ under 
the multiplication by $k$. This is illustrated for $k=2$ and $k=3$ in 
Fig.~\ref{fig:divb23}. 

\begin{figure}[htbp]
    \centering
 \includegraphics[scale=0.9]{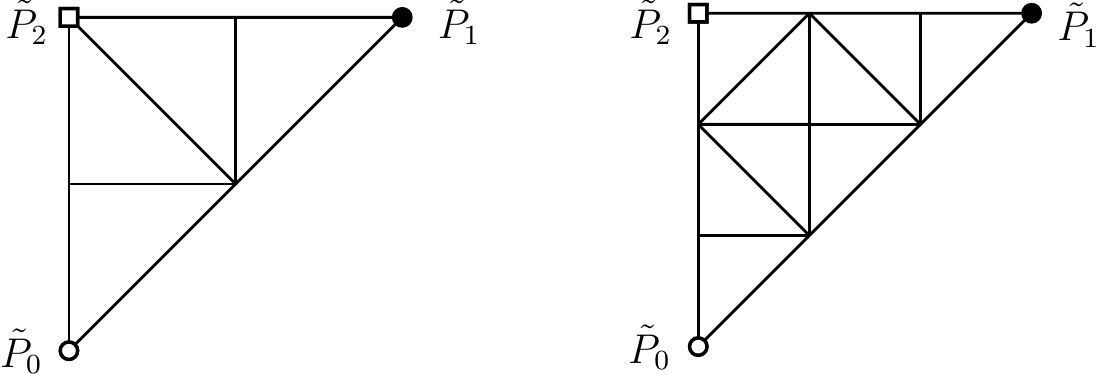}
    \caption{The subtriangles $T_j$ of $R_{B_2}$ for $k=2$ and $k=3$.}
    \label{fig:divb23}
\end{figure}

Consider the inverse map from $T_j$ to $R_{B_2}$ which is division by $k$ 
together with a suitable linear translation. Being a continuous map there exists 
at least one fixed point of this map. Moreover there is at most one such point 
in each $T_j$ because of the linearity. 

It remains to see that these points are distinct from each other. Such a 
repetition can occur only at the boundaries of the subtriangles $T_j$. However 
the multiplication by $k$ maps such a boundary to the boundary of $R_{B_2}$. 
The triangles meet at $k$ division points and they are mapped to one of the 
corner points. On the other hand a corner point $\tilde{P}_i$, that is fixed 
under multiplication by $k$, lies in only one of the triangles $T_j$. To 
see this, observe that $P_2$ is fixed under $\mc{B}_k$ if and only if $k$ is 
odd. In that case $\tilde{P}_2$ is contained in only one of the subtriangles 
$T_j$.
\end{proof}

Now, we collect several results we have proved so far and give the following 
correspondence between $ \Fix(\mc{B}_q)$ and $\F_q^2$.
\begin{lemma}\label{corresb2}
Let $q$ be a power a prime $p$ and let $k\geq 2$ be a fixed integer. Consider 
the cyclotomic number field $K=\Q(\zeta_{q^4-1})$ which contain the coordinates 
of the points fixed under $\mc{B}_k$. Let $\mf{p}$ be a prime ideal of $K$ 
lying over $p$. Then there exists a one-to-one correspondence
\[ \Fix(\mc{B}_q) \longleftrightarrow \F_q^2 \]
which is given by the reduction modulo $\mf{p}$.
\end{lemma}

\begin{proof}
There are $q^2$ fixed points of $\mc{B}_q$. We have $\mc{B}_q \equiv (x^q,y^q) 
\pmod{p}$ by Lemma~\ref{lemmab2}. The reduction of each point 
$(x,y)\in\Fix(\mc{B}_q)$ modulo $\mf{p}$ gives a different solution of 
the equations $x^q-x\equiv 0\pmod{p}$ and $y^q-y\equiv 0\pmod{p}$.
\end{proof}

This correspondence is compatible with the action of $\mc{B}_k$. If $\alpha = 
\varPhi_{B_2}(\sigma,\tau)$ is fixed under $\mc{B}_q$, then its coordinates are 
algebraic integers. Note that $\mc{B}_k(\alpha)$ is fixed under $\mc{B}_q$ too. 
Let $x\mapsto \overline{x}$ be the reduction map of the theorem. Then we have
\[ \overline{\mc{B}_k(\alpha)} = \mc{B}_k(\overline{\alpha}). \]
This characterization of $\F_q^2$, which is compatible with the action of 
$\mc{B}_k$, allows us to obtain the main result of this section.
\begin{theorem}\label{mainb2}
 The bivariate polynomial mapping $\mc{B}_k$ associated with $B_2\cong 
C_2$ induces a permutation of $\F_q^2$ if and only if $\gcd(q^4-1,k)=1$.
\end{theorem}
\begin{proof}
The theorem is easily verified for $k=0$ and $k=1$. Suppose that $k\geq2$ and 
$\gcd(q^4-1,k)=1$. In order to see that $\bar{\mc{B}}_k(x,y)$ is a permutation 
of $\F_q^2$, it is enough to see that $\mc{B}_k$ permute $\Fix(\mc{B}_q)$. By 
Theorem~\ref{fixb2}, the set $\Fix(\mc{B}_q)$ is explicit. Its elements 
$\varPhi_{B_2}(\sigma,\tau)$ come with rational $\sigma$ and $\tau$ whose 
denominators are relatively prime to $k$. Thus the map $\mc{B}_k$ permutes 
$\Fix(\mc{B}_q)$.

To see the converse, suppose that $\gcd(q^4-1,k)\neq 1$. Then there exist an 
integer $m>1$ dividing either $q^2-1$ or $q^2+1$. It follows that either 
$\varPhi_{B_2}(1/(q^2- 1),q/(q^2-1))$ or
$\varPhi_{B_2}(1/(q^2+ 1),q/(q^2+1))$ is not in 
\begin{align*}
 \mc{B}_k(\Fix(\mc{B}_q))= &\left\{ \varPhi_{B_2}\left( \frac{kd}{q\pm1}, 
\frac{ke}{q\pm1} \right):d,e\in\Z \right\} \\
& \cup \left\{ \varPhi_{B_2}\left( 
\frac{kd}{q^2\pm1}, \frac{\pm kqd}{q^2\pm1} \right):d\in\Z \right\}.
\end{align*}
Thus $\bar{\mc{B}}_k(x,y)$ is not surjective and as a result it is not a 
permutation.
\end{proof}

Now, we give a counterexample to the conjecture posed by Lidl and Wells in 
\cite{lidlwells}. The bivariate map $\bar{\mc{B}}_{13}$ is a permutation 
of $\F_p^2$ for an infinite number of primes by Theorem~\ref{mainb2}. More 
precisely $\bar{\mc{B}}_{13}: \F_p^2\rightarrow\F_p^2$ is a permutation if and 
only if $p \not\equiv 1,5,8,12\pmod{13}$. Suppose that $\bar{\mc{B}}_{13}$ is 
a composition of linear polynomial vectors and the generalized Chebyshev 
polynomials $g(2,k,b)$ of Lidl and Wells. Each occurrence of $g(2,k,b)$ will 
put 
a restriction on $p$, see Theorem~\ref{lidlwellsmain}. However it is not 
possible to obtain the set of primes $p \not\equiv 1,5,8,12\pmod{13}$
by the conditions $\gcd(k,p^s-1)=1$ for $s=1,2,3$. Thus $\bar{\mc{B}}_{13}$ 
cannot be expressed as a composition of linear polynomial vectors and 
polynomial vectors $g(k,n,b)$ where $k$ and $b$ are various integers.

\section{The family associated with $G_2$}\label{sec:G}
We refer to \cite{lie} for a nice introduction to the theory of Lie algebras. 
Let $\{\alpha_1,\alpha_2\}$ be a choice of simple roots for the Lie algebra 
$G_2$ with Cartan matrix 
\[\left[\begin{array}{cc}2 & -1\\-3 & 2 \end{array}\right].\]
The transpose of this matrix transforms the fundamental weights into the 
fundamental roots. We have
\begin{align*}
 \alpha_1&=2\omega_1-3\omega_2,\\
 \alpha_2&=-\omega_1+2\omega_2.
\end{align*}
The function $\varPhi_{G_2}=(\varphi_1,\varphi_2)$ of Theorem~\ref{veselov} is 
obtained by the action of the Weyl group on the fundamental weights $\omega_1$ 
and $\omega_2$. The functions $\varphi_1$ and $\varphi_2$ turn out to 
be
\begin{align*}
 \varphi_1(\sigma,\tau) = &\ e^{2\pi i \sigma}+e^{2\pi i \tau}+e^{2\pi 
i(\sigma+\tau)}+e^{-2\pi i\sigma}+e^{-2\pi i \tau}+e^{-2\pi 
i(\sigma+\tau)}, \\
 \varphi_2(\sigma,\tau) = &\ e^{2\pi i (2\sigma+\tau)}+e^{2\pi i 
(\sigma+2\tau)}+e^{2\pi 
i(\sigma-\tau)} \\ &\ + e^{-2\pi i (2\sigma+\tau)}+e^{-2\pi i 
(\sigma+2\tau)}+e^{-2\pi 
i(\sigma-\tau)}. 
\end{align*}
For each $(\sigma,\tau)\in\R^2$, we can simply write
\begin{align*}
  \Phi_{G_2}(\sigma,\tau) =\ &(2\cos(2\pi \sigma)+2\cos(2\pi \tau)+2\cos(2\pi 
(\sigma+\tau)),\\&2\cos(2\pi (2\sigma+\tau)+2\cos(2\pi 
(\sigma+2\tau))+2\cos(2\pi 
(\sigma-\tau))).
\end{align*}
Hofmann and Withers call this map the generalized cosine function for the 
underlying Lie algebra \cite{hoffwith}. Theorem~\ref{veselov} implies that 
there are bivariate polynomial mappings $P_{G_2}^k$, determined from the 
conditions $\varPhi_{G_2}(k\mathbf{x}) = P_{G_2}^k(\varPhi_{G_2}(\mathbf{x}))$ 
where $\mathbf{x}=(\sigma,\tau)$. For simplicity, let us put 
\[ \mc{G}_k := P_{G_2}^k. \]
These maps satisfy the composition property $\mc{G}_{kl} = \mc{G}_{k} \circ 
\mc{G}_{l} = \mc{G}_{l} \circ \mc{G}_{k}$ by their definition. The first few 
examples of these polynomials are:
\begin{align*}
 \mc{G}_0(x,y)=&\ (6,6),\\
 \mc{G}_1(x,y)=&\ (x,y),\\
 \mc{G}_2(x,y)=&\ (x^2 - 2x + (-2y - 6), -2x^3 + (6y + 18)x + (y^2 + 10y + 
18)),\\
 \mc{G}_3(x,y)=&\ (x^3 + (-3y - 9)x + (-6y - 12),(-3y - 6)x^3 \\ &+ (9y^2 + 
45y + 54)x+ (y^3 + 18y^2 + 63y + 60)),\\
 \mc{G}_4(x,y)=&\ (x^4 + (-4y - 10)x^2 + (-4y - 8)x + (2y^2 + 8y + 6), \\
&\ 2x^6 + (-12y - 36)x^4 + (-4y^2 - 28y - 40)x^3 \\& +(18y^2 + 108y +162)x^2 + 
(12y^3 + 120y^2 + 372y + 360)x \\ &+ (y^4 + 24y^3 + 134y^2 + 280y + 198)), \\
 \mc{G}_5(x,y)=&\ (x^5 + (-5y - 15)x^3 + (-5y - 10)x^2 + (5 y^2 + 35 y + 55) x 
\\&+ (10 y^2 + 50 y + 60), (5 y + 10) x^6 + (-30 y^2 - 150 y - 180) x^4 
\\& + (-5y^3 - 65 y^2 - 205 y - 190) x^3 + (45 y^3 + 360 y^2 + 945 y + 810) x^2 
\\& + (15 y^4 + 240 y^3 + 1200 y^2 + 2415 y + 1710) x 
\\&+ (y^5 + 30 y^4 + 255 y^3+ 920 y^2 + 1495 y + 900)).
\end{align*}
There is a recurrence relation satisfied by these maps from which it is 
straightforward to calculate further $\mc{G}_k$ \cite{withers}. If 
$\mc{G}_k=(f_k,g_k)$, then
\begin{align*}
 f_{k+6}=&\ x(f_{k+5}+f_{k+1})-(x+y+3)(f_{k+4}+f_{k+2})\\
 &+(x^2-2y-4)f_{k+3}-f_k, \\
 g_{k+6}=&\ y(g_{k+5}+g_{k+1})-(x^3-3xy-9x-5y-9)(g_{k+4}+g_{k+2})\\
 &+(y^2-2x^3+6xy+18x+12y+8)g_{k+3}-g_k.
\end{align*}

Let $q$ be a power of a prime $p$. Note that $\mc{G}_q \equiv (x^q,y^q) 
\pmod{p}$ for $q=2,3,4$ and $5$. We will show that this is true in general.

Let $\phi(t_1,t_2,t_3)=(t_1+1/t_1,t_2+1/t_2,t_3+1/t_3)$ and $\psi=(\sigma_1, 
\sigma_2, \sigma_3)$ where $\sigma_i$ is the $i$th elementary symmetric 
function. There exists $G_k(x,y,z)$ so that the following diagram commutes:
\begin{center}
\begin{tikzpicture}
  \matrix (m) [matrix of math nodes,row sep=1.5em,column sep=18em,minimum 
width=2em] {
     {\C^*}^3 & {\C^*}^3 \\
     \C^3 & \C^3 \\
     \C^3 & \C^3 \\};
  \path[-stealth]
    (m-1-1) edge node [left] {$\phi$} (m-2-1)
            edge node [above] 
{$(t_1,t_2,t_3)\mapsto(t_1^k,t_2^k,t_3^k)$} (m-1-2)
    (m-2-1) edge node [left] {$\psi$} (m-3-1)
    edge node [above] {$(u_1,u_2,u_3)\mapsto(D_k(u_1),D_k(u_2),D_k(u_3))$} 
(m-2-2)
    (m-1-2) edge node [right] {$\phi$} (m-2-2)
    (m-3-1) edge node [above] {$(x,y,z)\mapsto G_k(x,y,z)$} (m-3-2)
    (m-2-2) edge node [right] {$\psi$} (m-3-2);
\end{tikzpicture}\end{center} 
Commutativity of the upper part follows from the definition of Dickson 
polynomials. The existence of $G_k$ and the fact that each component of $G_k$ is 
in $\Z[x,y,z]$ follows from the fundamental theorem on symmetric polynomials.

\begin{lemma}
 If $t_1t_2t_3=1$ and $\psi(\phi(t_1,t_2,t_3))=(x,y,z)$, then $z=x^2-2y-4$.
\end{lemma}
\begin{proof}
 The proof is by direct computation. Put $t_3=1/(t_1t_2)$. We have
 \begin{align*}
  x&=\frac{(t_2^2 + t_2)t_1^2 + (t_2^2 + 1)t_1 + t_2 + 1}{t_1t_2}\\
  y&=\frac{t_2^3t_1^4 + (t_2^4 + t_2^3 + t_2^2 + t_2)t_1^3 + (t_2^3 + t_2)t_1^2 
+ (t_2^3 + t_2^2 + t_2 + 1)t_1 + t_2}{t_2^2t_1^2}\\
 z&=\frac{(t_2^4 + t_2^2)t_1^4 + (t_2^4 + 2t_2^2 + 1)t_1^2 + t_2^2 + 
1}{t_2^2t_1^2}.
\end{align*}
One can verify that $z=x^2-2y-4$.
\end{proof}

Define $V=\{ (x,y,z)\in\C^3: z=x^2-2y-4\}$. The map $G_k$ induces a map on $V$ 
because $t_1t_2t_3=1$ implies that $t_1^kt_2^kt_3^k=1$. Let $\pi: V \rightarrow 
\C^2$ be the projection to the first two components, i.e. $\pi(x_1,x_2,x_3) = 
(x_1,x_2)$. We define the bivariate map $\tilde{\mc{G}}_k$ by 
the following commutative diagram.
\begin{center}
\begin{tikzpicture}
  \matrix (m) [matrix of math nodes,row sep=1.5em,column sep=14em,minimum 
width=2em] {
     V & V \\
     \C^2 & \C^2 \\};
  \path[-stealth]
(m-1-1) edge node [left] {$\pi$} (m-2-1)
        edge node [above] {$(x,y,z)\mapsto G_k(x,y,z)$} (m-1-2)
(m-2-1) edge node [above] {$(x,y)\mapsto \tilde{\mc{G}}_k(x,y)$} (m-2-2)
(m-1-2) edge node [right] {$\pi$} (m-2-2);
\end{tikzpicture}\end{center} 
A formula for $\tilde{\mc{G}}_k(x,y)$ is obtained by replacing $z$ with 
$x^2-2y-4$ in the first two components of $G_k$. More precisely 
$\tilde{\mc{G}}_k(x,y)= \pi(G_k(x,y,x^2-2y-4))$. For example 
\[ \tilde{\mc{G}}_k(x,y)=(x^2 + (-2y - 6), -2x^3 - 4x^2 + (4y + 8)x + 
(y^2 + 8y + 12)). \]
The bivariate maps $\tilde{\mc{G}}_k$ and $\mc{G}_k$ are conjugates to each 
other. To see this relation, let us consider
\begin{align*}
a &=e^{2\pi i \sigma}+e^{-2\pi i \sigma},\\ b&=e^{2\pi i 
\tau}+e^{-2\pi i \tau},\\ c&=e^{2\pi i (\sigma+\tau)}+e^{-2\pi i (\sigma+\tau)}.
\end{align*}
Then $\varphi_1=a+b+c$ and $\varphi_2=ab+ac+bc-a-b-c$. In other words 
$\varphi_1=\sigma_1$ and $\varphi_2=\sigma_2-\sigma_1$. Let 
$L:(x,y)\mapsto(x,y-x)$. We have
\[\tilde{\mc{G}}_k=L\circ\mc{G}_k\circ L^{-1}.\]
Thus $\tilde{\mc{G}}_k$ and $\mc{G}_k$ are conjugate to each other by the linear 
map $L$. Now, we prove a lemma that is key to obtain the correspondence between 
$\Fix(\mc{G}_q)$ and $\F_q^2$.

\begin{lemma}
 Let $q$ be a power of a prime $p$. Then
 \begin{enumerate}
  \item $\tilde{\mc{G}}_q(x,y)=(x^q,y^q) \pmod{p}$,
  \item $\mc{G}_q(x,y)=(x^q,y^q) \pmod{p}$.
 \end{enumerate}
\end{lemma}
\begin{proof}
It is enough to prove the first assertion because if that is the case then we 
have
\begin{align*}
 \mc{G}_q(x,y) &\equiv (L^{-1} \circ \tilde{\mc{G}}_k \circ L) (x,y) \pmod{p} \\
 &\equiv L^{-1}(x^q,(y-x)^q) \pmod{p} \\
 &\equiv (x^q,y^q) \pmod{p}.
\end{align*}

By the fundamental theorem on symmetric polynomials, there exists a function 
$F_k\in\Z[x,y,z]$ such that the following diagram commutes:
\begin{center}
\begin{tikzpicture}
  \matrix (m) [matrix of math nodes,row sep=1.5em,column sep=14em,minimum 
width=2em] {
     {\C^*}^3 & {\C^*}^3 \\
     \C^3 & \C^3 \\};
  \path[-stealth]
(m-1-1) edge node [left] {$\psi$} (m-2-1)
        edge node [above] {$(t_1,t_2,t_3)\mapsto(t_1^k,t_2^k,t_3^k)$} (m-1-2)
(m-2-1) edge node [above] {$(x,y,z)\mapsto F_k(x,y,z)$} (m-2-2)
(m-1-2) edge node [right] {$\psi$} (m-2-2);
\end{tikzpicture}\end{center} 
For example $F_2(x,y,z)=(x^2-2y,y^2-2xz,z^2)$. Recall that $D_q(x) = x^q 
\pmod{p}$. It follows that $G_q(x,y,z) =F_q(x,y,z) \pmod{p}$. Let $\pi_1$ and 
$\pi_2$ be the projections to the first and second components, respectively. 
Lidl and Wells provide explicit formulas for $\pi_1(F_k)$ and $\pi_2(F_k)$. 
More precisely
\begin{align*}
 \pi_1(F_k)&=\sum_{i=0}^{\lfloor k/2\rfloor} \sum_{j=0}^{\lfloor k/3\rfloor} 
\frac{k(-1)^i}{k-i-2j}\binom{k-i-2j}{i+j}  \binom{i+j}{i}x^{k-2i-3j}y^iz^j \\
 \pi_2(F_k)&=\sum_{i=0}^{\lfloor k/2\rfloor} \sum_{j=0}^{\lfloor k/3\rfloor} 
\frac{k(-1)^j}{k-i-2j}\binom{k-i-2j}{i+j}  \binom{i+j}{i}x^iy^{k-2i-3j}z^{i+2j}
\end{align*}
where only those terms occur for which $k\geq 2i+3j$ \cite{lidlwells}. It 
easily follows from these formulas that $\pi(F_q(x,y,x^2-2y-4))=(x^q,y^q) 
\pmod{p}$. Therefore $\tilde{\mc{G}}_q(x,y)=\pi(G_k(x,y,x^2-2y-4))=(x^q,y^q) 
\pmod{p}$.
\end{proof}

Let $k\geq 2$ be an integer. Similar to the set $\Delta_{B_2}$, the set of 
points with bounded orbits under $\mc{G}_k$ is the set
\[\Delta_{G_2}=\{\varPhi_{G_2}(\sigma,\tau):\sigma,\tau\in\R\}.\]
As a result, a point that is fixed under $\mc{G}_k:\C^2 \rightarrow \C^2$ is of 
the form $\varPhi_{G_2}(\sigma,\tau)$ for some $\sigma,\tau\in \R$. The set 
$\Delta_{G_2}$, which is shown in Fig.~\ref{fig:dg2}, is contained in $\R^2$.
There are three corner points, namely $Q_0=(6,6), Q_1 = (-3,6)$ and 
$Q_2=(-2,-2)$. The region $\Delta_{G_2}$ is enclosed by the singular cubic curve 
$(y+6x+12)^2 = 4(x+3)^3$ and the parabola $4y=x^2-12$. The node of the singular 
cubic curve is at $Q_1$. 

\begin{figure}[htbp]
    \centering
 \includegraphics[scale=0.7]{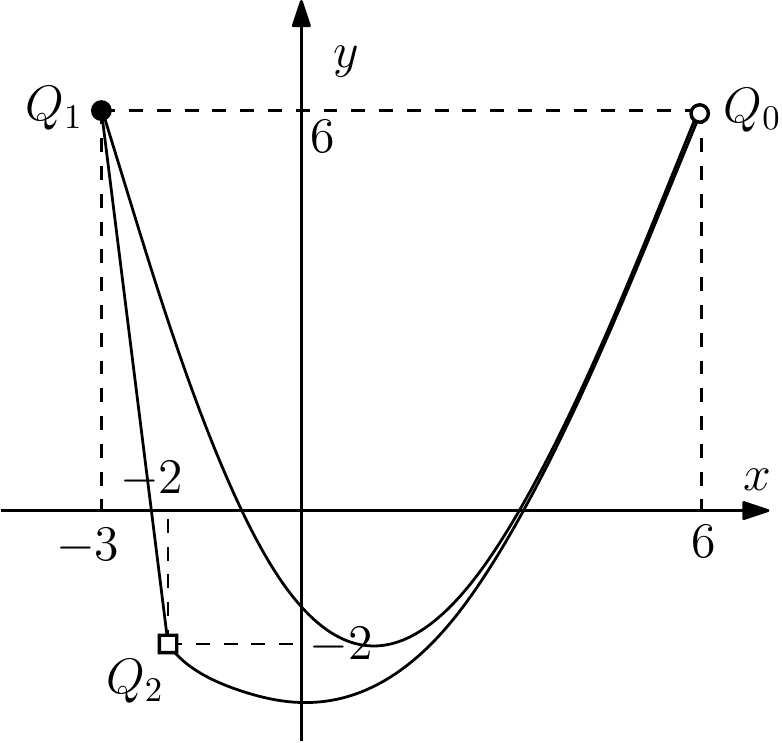}
    \caption{The set $\Delta_{G_2}$.}
    \label{fig:dg2}
\end{figure}

We want to find a fundamental region in $\sigma\tau$-plane whose elements are 
in one-to-one correspondence with the elements of $\Delta_{G_2}$ under 
$\varPhi_{G_2}$. If $(\sigma,\tau) \equiv (\sigma',\tau')\pmod{\Z^2}$, 
then it is easy to see that $\varPhi_{G_2}(\sigma,\tau) = 
\varPhi_{G_2}(\sigma',\tau')$. Thus it is enough to consider $0 \leq \sigma,\tau 
\leq 1$ to obtain any point in $\Delta_{G_2}$ under the map $\varPhi_{G_2}$. 
Moreover there are extra symmetries coming from the action of the Weyl group. 
Observe that $\varPhi_{G_2}(\sigma,\tau)$ is equal to any one of the following 
twelve expressions:
\[\begin{array}{ll|ll|ll}
\text{I} &\varPhi_{G_2}(\sigma,\tau) &\text{V} 
&\varPhi_{G_2}(\sigma,-\sigma-\tau) &\text{IX} 
& \varPhi_{G_2}(\tau,-\sigma-\tau)\\
\text{II} &\varPhi_{G_2}(\tau,\sigma) &\text{VI} 
&\varPhi_{G_2}(-\sigma-\tau,\sigma) 
&\text{X} &\varPhi_{G_2}(-\sigma-\tau,\tau)\\
\text{III} &\varPhi_{G_2}(-\sigma,-\tau) &\text{VII} 
&\varPhi_{G_2}(-\sigma,\sigma+\tau) & 
\text{XI} &\varPhi_{G_2}(-\tau,\sigma+\tau)\\
\text{IV} &\varPhi_{G_2}(-\tau,-\sigma) &\text{VIII} 
&\varPhi_{G_2}(\sigma+\tau,-\sigma) 
&\text{XII} &\varPhi_{G_2}(\sigma+\tau,-\tau)
\end{array}\]

Under these symmetries, the square $0 \leq \sigma,\tau \leq 1$ can be separated 
into twelve subtriangles. This is shown in Fig.~\ref{fig:RG2}. Define
\[ R_{G_2}=\{ (\sigma,\tau)\in\R^2 \mathrel| 0 \leq \sigma \leq 1/3 
\text{ and } \sigma \leq \tau \leq (1-\sigma)/2  \}. \]
Note that the restricted function $\varPhi_{G_2}:R_{G_2} \rightarrow 
\Delta_{G_2}$ is one-to-one and onto. Set $\tilde{Q}_0=(0,0), 
\tilde{Q}_1=(1/3,1/3)$ and $\tilde{Q}_2=(0,1/2)$. Then 
$\varPhi_{G_2}(\tilde{Q}_i)=Q_i$ for each $i\in\{1,2,3\}$. This correspondence 
(and more) is symbolized by the use of different marks, such as circles, disks 
and square. 

\begin{figure}[htbp]
    \centering
 \includegraphics[scale=0.80]{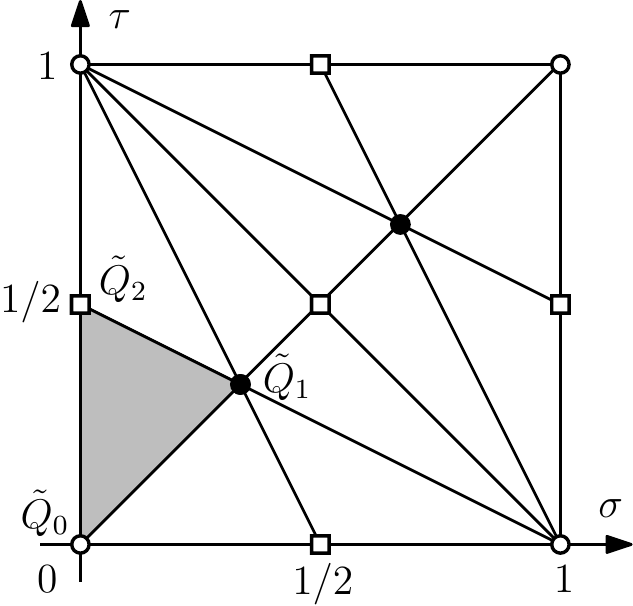}
    \caption{The fundamental region $R_{G_2}$.}
    \label{fig:RG2}
\end{figure}

Now, we are ready analyze the set of fixed points under the bivariate map 
$\mc{G}_k$. Let $k \geq 2$ be a fixed integer. Let $\alpha = 
\varPhi_{G_2}(\sigma,\tau)$ be a fixed point under $\mc{G}_k$. We want to 
determine the values of $\sigma$ and $\tau$ such that 
$\mc{G}_k(\varPhi_{G_2}(\sigma,\tau)) = \varPhi_{G_2}(k\sigma,k\tau) = 
\varPhi_{G_2} (\sigma , \tau)$. 
There are twelve possibilities. For example, let us consider 
$(k\sigma,k\tau) \equiv (-\sigma-\tau,\sigma)$ modulo $\Z^2$. This is the case 
VI. We have 
\[k^2\sigma \equiv -k\sigma -k\tau \equiv -k\sigma -\sigma \pmod{\Z}. \]
It follows that $\sigma=d/(k^2+k+1)$ for some integer $d$. Since $k\sigma 
\equiv -\sigma-\tau \pmod{\Z}$, we have $\tau\equiv-(k+1)\sigma \pmod{\Z}$. 
Thus
\[ \alpha = \varPhi_{G_2}\left( \frac{d}{k^2+k+1}, \frac{-(k+1)d}{k^2+k+1} 
\right)= \varPhi_{G_2}\left( \frac{d}{k^2+k+1}, \frac{kd}{k^2+k+1} 
\right).\]
Here, the second equality follows from $d+kd-(k+1)d\equiv 0 
\pmod{k^2+k+1}$ and the definition of $\varPhi_{G_2}$. In general, a fixed 
point $\alpha$ fits into one of the following sets:
\begin{align*}
 S_1 &= \left\{ \varPhi_{G_2}\left( \frac{d}{k-1}, \frac{e}{k-1} 
\right):d,e\in\Z \right\},\\
 S_2 &= \left\{ \varPhi_{G_2}\left( \frac{d}{k^2-1}, \frac{dk}{k^2-1} 
\right):d\in\Z \right\},\\
 S_3 &= \left\{ \varPhi_{G_2}\left( \frac{d}{k^2+k+1}, \frac{dk}{k^2+k+1} 
\right):d\in\Z \right\},\\
 S_4 &= \left\{ \varPhi_{G_2}\left( \frac{d}{k+1}, \frac{e}{k+1} 
\right):d,e\in\Z \right\},\\
 S_5 &= \left\{ \varPhi_{G_2}\left( \frac{d}{k^2-1}, \frac{d(-k)}{k^2-1} 
\right):d\in\Z \right\},\\
 S_6 &= \left\{ \varPhi_{G_2}\left( \frac{d}{k^2-k+1}, \frac{d(-k)}{k^2-k+1} 
\right):d\in\Z \right\}.
\end{align*}
The computation above shows that the fixed points of type VI takes place in 
$S_3$. One can do similar computations for the other cases and obtain the 
following table which gives the correspondence between the sets $S_i$ and 
the types of symmetries.
\[\begin{array}{|c|c|c|c|c|c|} \hline
S_1 & S_2 & S_3 & S_4 & S_5 & S_6\\ \hline
\text{I} & \text{II, V} & \text{VI, IX, X } & \text{III } & \text{IV, VII} & 
\text{VIII,XI,XII}\\ \hline
\end{array}\]

\begin{theorem}\label{fixg2}
Let $k \geq 2$ be a fixed integer. Then 
\[ \Fix(\mc{G}_k)=S_1 \cup S_2 \cup S_3 \cup S_4 \cup S_5 \cup S_6.\]
\end{theorem}

Note that the union $\cup S_i$ is not disjoint. For example the point $Q_0$ is 
contained in each $S_i$. The following theorem gives the cardinality of the set 
of points fixed under $\mc{G}_k$.
\begin{theorem}
Let $k\geq2$ be a fixed integer. Then $|\Fix(\mc{G}_k)|=k^2$.
\end{theorem}
\begin{proof}
We follow the idea of Uchimura \cite{uchimura}. The fundamental region 
$R_{G_2}$ is a closed bounded domain. Divide $R_{G_2}$ into $k^2$ subtriangles 
$T_1,\ldots,T_{k^2}$ such that each one of them is mapped onto $R_{G_2}$ under 
the multiplication by $k$. This is illustrated for $k=2$ and $k=3$ in 
Figure~\ref{fig:divg23}. 

\begin{figure}[htbp]
    \centering
 \includegraphics[scale=0.75]{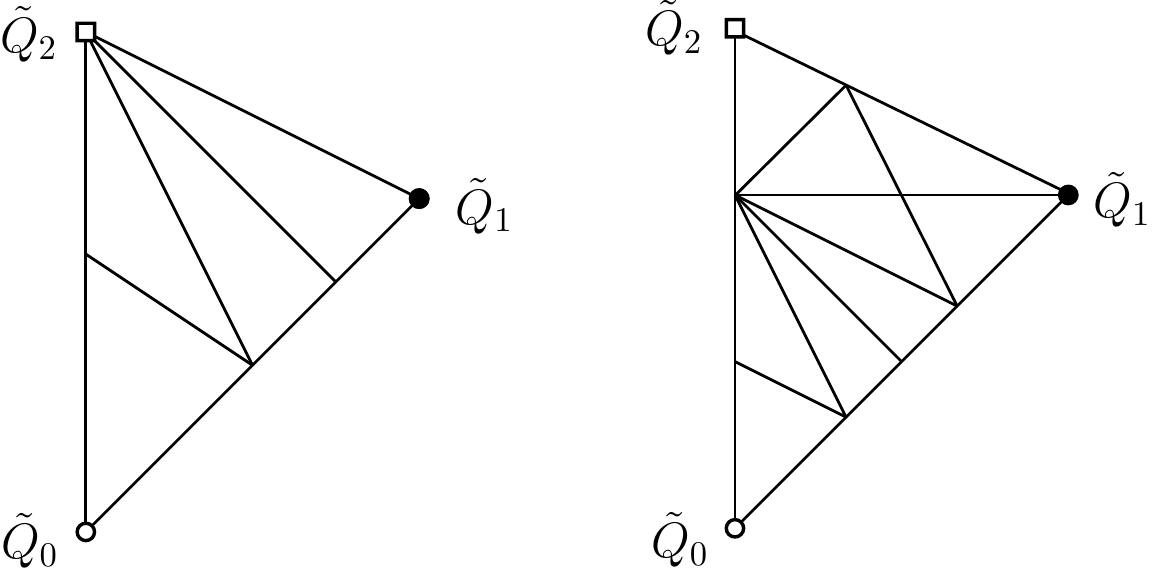}
    \caption{The subtriangles $T_j$ of $R_{G_2}$ for $k=2$ and $k=3$.}
    \label{fig:divg23}
\end{figure}

Note that $Q_1$ is fixed under $\mc{G}_k$ if and only if $k$ is not divisible 
by $3$. In that case $\tilde{Q}_1$ is contained in only one of the subtriangles 
$T_j$. Similarly $Q_2$ is fixed under $\mc{G}_k$ if and only if $k$ is odd. In 
such a case, the point $\tilde{Q}_2$ is contained in only one of the 
subtriangles $T_j$ as well. The rest of the proof is as the same as the proof 
of Theorem~\ref{uchib2}.
\end{proof}

The proof of the following lemma is similar to the proof of 
Lemma~\ref{corresb2} and omitted. 
\begin{lemma}
Let $q$ be a power a prime $p$ and let $k\geq 2$ be a fixed integer. Consider 
the cyclotomic number field $K=\Q(\zeta_{q^6-1})$ which contain the coordinates 
of the points fixed under $\mc{G}_k$. Let $\mf{p}$ be a prime ideal of $K$ 
lying over $p$. Then there exists a one-to-one correspondence
\[ \Fix(\mc{G}_q) \longleftrightarrow \F_q^2 \]
which is given by the reduction modulo $\mf{p}$.
\end{lemma}

Similar to the case associated with $B_2\cong C_2$, the correspondence given by 
this lemma is compatible with the action of $\mc{G}_k$. If $\alpha = 
\varPhi_{G_2} (\sigma,\tau)$ is fixed under $\mc{G}_q$, then $\mc{G}_k(\alpha)$ 
is fixed under $\mc{G}_q$ too. Moreover
\[ \overline{\mc{G}_k(\alpha)} = \mc{G}_k(\overline{\alpha}). \]
The characterization of $\F_q^2$, which is compatible with the action of 
$\mc{G}_k$, allows us to obtain the main result of this section.

\begin{theorem}\label{maing2}
 The bivariate polynomial mapping $\mc{G}_k$ associated with $G_2$ induces a 
permutation of $\F_q^2$ if and only if $\gcd(k,q^6-1)=1$.
\end{theorem}
\begin{proof}
The theorem is easily verified for $k=0$ and $k=1$. Suppose that $k\geq2$ and 
$\gcd(k,q^6-1)=1$. In order to see that $\bar{\mc{G}}_k(x,y)$ is a permutation 
of $\F_q^2$, it is enough to see that $\mc{G}_k$ permute $\Fix(\mc{G}_q)$. By 
Theorem~\ref{fixg2}, the set $\Fix(\mc{G}_q)$ is explicit. Its elements 
$\varPhi_{G_2}(\sigma,\tau)$ come with rational $\sigma$ and $\tau$ whose 
denominators are relatively prime to $k$. Thus the map $\mc{G}_k$ permutes 
$\Fix(\mc{G}_q)$.

To see the converse, suppose that $\gcd(k,q^6-1)\neq 1$. Let $n$ be an integer 
such that $n\in\{ q^2-1,q^2+1,q^2-q+1, q^2+q+1 \}$ and $\gcd(k,n)>1$. Consider 
the element 
\[ \alpha= \varPhi_{G_2}\left( \frac{1}{n},\frac{\pm q}{n} \right) \]
with a suitable choice of sign so that $\alpha\in \Fix(\mc{G}_q)$. The element 
$\alpha$ is in $\Fix(\mc{G}_q)$ but it is not in $\mc{G}_k(\Fix(\mc{G}_q))$. As 
a result, the map $\mc{G}_k$, restricted to $\Fix(\mc{G}_q)$, is not 
surjective. Thus $\bar{\mc{G}}_k(x,y)$ is not a permutation.
\end{proof}

Now, we give a counterexample to the conjecture posed by Lidl and Wells in 
\cite{lidlwells}. The bivariate map $\bar{\mc{G}}_{13}$ is a permutation 
of $\F_p^2$ for an infinite number of primes by Theorem~\ref{maing2}. More 
precisely $\bar{\mc{G}}_{13}: \F_p^2\rightarrow\F_p^2$ is a permutation if and 
only if $p \not\equiv 1,3,4,9,10,12\pmod{13}$. Suppose that $\bar{\mc{G}}_{13}$ 
is a composition of linear polynomial vectors and the generalized Chebyshev 
polynomials $g(2,k,b)$ of Lidl and Wells. Each occurrence of $g(2,k,b)$ will 
put a restriction on $p$, see Theorem~\ref{lidlwellsmain}. However it is not 
possible to obtain the set of primes $p \not\equiv 1,3,4,9,10,12\pmod{13}$
by the conditions $\gcd(k,p^s-1)=1$ for $s=1,2,3$. Thus $\bar{\mc{G}}_{13}$ 
cannot be expressed as a composition of linear polynomial vectors and 
polynomial vectors $g(k,n,b)$ where $k$ and $b$ are various integers.

{\small
\def\refname{References}
\newcommand{\etalchar}[1]{$^{#1}$}

\end{document}